\documentclass[11pt]{amsart}
\usepackage{amssymb,amsthm,amsmath}
\usepackage{graphics,psfrag,graphicx,subfigure,epsfig}
\usepackage{enumerate}
\usepackage[all]{xy}
\usepackage{mathtools} 
\usepackage{mathtools} 
\usepackage{appendix} 
\usepackage{euscript}
\usepackage{eufrak}
\def\f{\mathcal{F}}
\def\fo{\mathcal{F}}
\def\g{\mathcal{G}}
\def\S{\mathcal{S}}
\def\fin{\mathcal{FIN}}
\def\vc{\mathcal{VC}}
\def\vci{\mathcal{VC}^\infty}
\def\C{\mathcal{C}^\infty}

\def\evc{\underline{\underline{E}}G}

\def\efin{\underline{E}G}
\def\efine{\underline{E}}

\def\z{\mathbb{Z}}

\DeclareMathOperator{\vcd}{vcd}
\DeclareMathOperator{\Mod}{Mod}
\DeclareMathOperator{\PMod}{PMod}
\DeclareMathOperator{\Aut}{Aut}

\DeclareMathOperator{\gdvc}{\underline{\underline{gd}}}

\def\gd{\mathrm{gd }}
\def\gdf{\underline{\mathrm{gd}}}

\def\rhos{\rho_{\sigma}}

\newtheorem{thm}{Theorem}[section]
\newtheorem{prop}[thm]{Proposition}

\theoremstyle{definition}

\newtheorem{rem}[thm]{Remark}

\newtheorem{lem}[thm]{Lemma}
\newtheorem{cor}[thm]{Corollary}

\title[Virtually-cyclic dimension of mapping class groups]{On the virtually-cyclic dimension of mapping class groups of punctured spheres}

\author{J. Aramayona }
\address{Universidad Aut\'onoma de Madrid \& ICMAT}
 \email{aramayona@gmail.com}
 
\author{D. Juan-Pineda}
\address{Centro  de  Ciencias Matem\'aticas. \\
Universidad Nacional Aut\'onoma de M\'exico, Campus  Morelia\\
Ap.Postal  61-3 Xangari\\ Morelia, Michoac\'an.  M\'EXICO 58089}
 \email{daniel@matmor.unam.mx}

\author{A. Trujillo-Negrete}
\address{Centro  de  Ciencias Matem\'aticas. \\
Universidad Nacional Aut\'onoma de M\'exico, Campus  Morelia\\
Ap.Postal  61-3 Xangari\\ Morelia, Michoac\'an.  M\'EXICO 58089}
 \email{aletn@matmor.unam.mx}

\begin{document}
\begin{abstract}
We calculate the virtually-cyclic dimension of the
mapping class group of a sphere with at most six punctures. As an immediate consequence, we obtain the virtually-cyclic dimension of the mapping class group of the twice-holed torus and of the closed genus-two surface. 

For spheres with an arbitrary number of punctures, we give a new upper bound for the virtually-cyclic dimension of their mapping class group, improving the recent bound of Degrijse-Petrosyan \cite{DP}.
\end{abstract}
\maketitle
\section{Introduction} 

 Given a discrete group $G$, a family $\f$ of subgroups of $G$ is a set of subgroups of $G$ which is closed under conjugation and taking subgroups. Of particular interest here are the families $\fin_G$ and $\vc_G$ which consist, respectively, of finite and virtually cyclic subgroups of $G$.  
 
 A model for the \textit{classifying space $E_{\fo}G$ of the family $\fo$} is a $G$-CW-complex $X$, such that the fixed point set of $H\in \f$ is contractible, and is empty whenever $H\notin \f$. Using standard terminology, we denote  $E_{\fin}G$ by $\efin$,  and $E_{\vc}G$ by $\evc$.  The study of models for these families finds a large part of its motivation in  the Baum-Connes and  Farrell-Jones  conjectures, respectively. 
 
 Although a model for the space $E_{\fo}G$ always exists, it need not be finite dimensional. The smallest possible dimension of a model of $E_{\f}G$ is called the \textit{geometric dimension of $G$ for the family} $\f$, and is usually denoted  $\gd_{\f}G$. 
Again using standard terminology, we will write $\gdf G=\gd_{\fin}G$ and $\gdvc G=\gd_{\vc}G$, and refer to them as the {\em proper geometric dimension} and the {\em virtually-cyclic dimension} of $G$, respectively.  
For many families of groups these two numbers are related by the inequality \begin{equation}
\gdvc G \le \gdf G + 1;
\label{eq:main}
\end{equation} although it is known to not be true in general  \cite[Example 6.5]{DP2}. 
Classes of groups for which it does hold include:  ${\rm CAT}(0)$ groups \cite{luck-cat(0)-vc}, hyperbolic groups \cite{Daniel-Leary}, standard braid groups \cite{FG}, and groups satisfying a certain property (Max) \cite[Theorem 5.8]{luck-weiermann}, which roughly states that every infinite virtually-cyclic subgroup is contained in a unique maximal such subgroup (see Section \ref{sec:general}). 

In this note we investigate the relation between  $\gdvc G$ and $\gdf G$ for the mapping class group $\Mod(S)$ of a connected, orientable surface $S$, mainly in the case when $S$ has genus zero. We stress that mapping class groups do not fall in any of the categories above; however,  they contain finite-index subgroups with property (Max) \cite[Prop. 5.1]{Daniel-Ale}; compare with Lemma \ref{lem-max-0} below. For these subgroups the inequality \eqref{eq:main} holds, although this does not say anything about whether this is the case for the whole group. 

We will denote by $S_{g,b}^n$ the connected orientable surface of genus $g$, with $b$ boundary components and  $n$ punctures. If $b=0$, we will omit $b$ from the notation.  Our main result is as follows: 

\begin{thm}
Let $n\in \{5,6\}$. Then $\gdvc \Mod(S_0^n) = \gdf \Mod(S_0^n) + 1$. 
\label{thm:main}
\end{thm}

We remark that $\Mod(S_0^n)$ is finite for $n\le 3$, and virtually-free if $n=4$ \cite{margalit}. As an immediate corollary of Theorem \ref{thm:main}, we will obtain: 

\begin{cor}
If $S\in \{S_1^2, S_2^0\}$, then $\gdvc \Mod(S) = \gdf \Mod(S) + 1$.
\label{cor:positgen}
\end{cor}

At this point, we remark that the proper geometric dimension of $\Mod(S_g^n)$ is known \cite{AMP} to coincide with its virtual cohomological dimension, which in turn was computed by Harer \cite{Harer} and is an explicit linear function of $g$ and $n$ (in the particular case when $g=0$, it is equal to $n-3$).

\medskip

As a further corollary of Theorem \ref{thm:main} we calculate the exact value of the virtually-cyclic dimension of the spherical braid group $B_n$ on $n$ strands, for $n\in \{5,6\}$. Indeed, using the classical fact  that $B_n$ is a finite extension of $\Mod(S_0^n)$, we will obtain: 
\begin{cor}
If $n\in \{5,6\}$, then $\gdvc(B_n) = \gdf(B_n) + 1$.
\label{cor:braids}
\end{cor}

This latter result should be compared with a recent theorem of Flores and Gonz\'alez-Meneses \cite{FG}, which proves the analogous statement for braid groups of the disk, with an arbitrary number of strands.

In order to prove Theorem \ref{thm:main}, we will use a result of L\"uck-Weiermann \cite{luck-weiermann} (stated as Theorem \ref{thm-luck-weierm} below) which relates the virtually-cyclic dimension of a group $G$ to the proper dimension of certain subgroups associated to infinite-order elements of $G$. We then use the Nielsen-Thurston classification of mapping classes and a case-by-case analysis, to bound the dimension of such subgroups. 

\begin{rem}
The argument used in the proof of Theorem \ref{thm:main} will not generalize to spheres with an arbitrary number of punctures; see Remark \ref{rem:casuistic2} below for more details. 
In spite of this, the interested reader can check that an immediate adaptation of the proof of Theorem \ref{thm:main} for $n=6$ gives a direct proof of Corollary \ref{cor:positgen}, as well as of the analogous statement for $\Mod(S_1^3)$. In particular, we obtain that inequality \eqref{eq:main} is in fact an equality for all surfaces $S_g^n$ for which $3g - 3 + n \le 3$. 
\label{rem:casuistic}
\end{rem}

\medskip

For a general number of punctures, a recent result of Degrijse-Petrosyan \cite{DP} gives a bound for $\gdvc \Mod(S_g^n)$ which is linear in $g$ and $n$ ; see Proposition \ref{prop:DP} below. In the particular case when $g=0$, this bound takes the form:  

\begin{equation}
\gdvc \Mod(S_0^n) \le 3n -8 = 3 \cdot \gdf \Mod(S_0^n) +1 
\end{equation}

Using the aforementioned result of L\"uck-Weiermann \cite{luck-weiermann} with a theorem of Cameron-Solomon-Turull \cite{Cameron}, we will prove the following slightly improved bound:

\begin{thm}
Suppose $n\geq 4$. Let $b_n$ be the number of ones in the  binary expression of $n$. Then
\begin{equation*}
\gdvc \Mod(S_0^n)\leq  n-4+\left[\frac{3n-1}{2}\right]-b_n,  
\end{equation*} where $[\cdot]$ denotes integer part. 
\label{propo:bound}
\end{thm}

\begin{rem} Observe that $3n - 8 \geq n-4+\left[\frac{3n-1}{2}\right]-b_n$ for all $n\ge 4$, and that the inequality is strict for $n\ge 5$. 
\end{rem}

\noindent{\bf Pure mapping class groups of spheres}. As we will observe in Lemma \ref{lem-max-0} pure mapping class groups of spheres have property (Max), and hence inequality \eqref{eq:main} is satisfied. Moreover, we will remark in Proposition \ref{prop-gdvc-mcgn-1} that in this case we get an equality, in fact. 

\medskip

\noindent{\bf Surfaces with boundary.} It follows from the definition that $\Mod(S)$ is torsion-free whenever $S$ has boundary. Combining this with a number of results by various authors, quickly yields that  \eqref{eq:main} holds for surfaces with boundary. The argument is essentially contained in the paper by Flores and Gonz\'alez-Meneses \cite{FG}; we offer a short account in the Appendix.

\bigskip

\noindent{\bf Acknowledgements.} J. A. was partially supported by grants RYC-2013-13008 and MTM2015-67781. D. Juan-Pineda and A. Tru\-ji\-llo-Negrete were partially supported by CONCAYT FORDECYT 265667. We would like to thank Yago Antol\'in, John Guaschi and Conchita Mart\'inez for conversations.


\section{Mapping class groups and braid groups} 
\label{sec:defs} 
In this section we give some preliminaries on mapping class groups, and their relation with braid groups. We refer the reader to \cite{margalit} for a thorough discussion on these and related topics.

\subsection{Mapping class groups} Let $S$ be a (possibly disconnected) orientable surface with empty boundary  and negative Euler characteristic, so that it supports a complete hyperbolic metric of finite area. Sometimes it will be convenient to regard (some of) the punctures of $S$ as marked points, and we will switch between the two points of view without further mention. 
As mentioned above, we will write $S_g^n$ to denote the connected surface of genus $g$ with $n$ marked points. 

The {\em mapping class group} $\Mod(S)$ is the group of isotopy classes of self-homeomorphisms of $S$; elements of $\Mod(S)$ are called {\em mapping classes}. The {\em pure mapping class group} $\PMod(S)$ is the subgroup of $\Mod(S)$ whose elements send every marked point to itself; observe that $\PMod(S)$ has finite index in $\Mod(S)$. 

Since we will deal mainly with surfaces of genus zero, from now on we will restrict our attention to the case of $S=S_0^n$, with $n\ge 3$.

\subsubsection{Curves and multicurves} By a {\em curve} on $S_0^n$ we mean the (free) isotopy class of a simple closed curve that does not bound a disk with at most one marked point. A {\em multicurve} is then a set of curves that {\em pairwise disjoint}, i.e. they may be realized in a disjoint manner on $S_0^n$. An easy counting argument shows that a maximal multicurve on $S_0^n$ has $n-3$ elements. 

\subsubsection{Nielsen-Thurston classification} We say that $f \in \Mod(S_0^n)$ is {\em reducible} if there exists a multicurve $\sigma \subset S_0^n$ such that $f(\sigma)=\sigma$; otherwise, we say that $f$ is {\em irreducible}. A notable example of a reducible element is the {\em Dehn twist} $T_\alpha$ about the curve $\alpha$; see \cite{margalit} for definitions and properties of Dehn twists.  Finally, we note that finite-order elements of $\Mod(S_0^n)$ may be reducible or irreducible. 

The celebrated {\em Nielsen-Thurston classification} of mapping classes asserts that an irreducible element of infinite order has a representative which is a {\em pseudo-Anosov} homeomorphism; see \cite[Ch. 5]{margalit} for details. For this reason, irreducible mapping classes of infinite order are normally referred to as {\em pseudo-Anosov} mapping classes. 

\subsubsection{Canonical reduction system}  Note that, in general, a reducible mapping class may fix more than one multicurve. For this reason, we define the {\em canonical reduction system} of a mapping class as the intersection of all the multicurves that it fixes. For instance,  the canonical reduction system of the Dehn twist $T_\alpha$ is equal to $\alpha$.

\subsubsection{The cutting homomorphism} Let $\sigma$ be a multicurve on $S_0^n$, and consider $(\Mod(S_0^n))_\sigma=\{g \in\Mod(S_0^n) | g(\sigma)=\sigma \}$. 
Denote by $S_0^n - \sigma$ the (disconnected) surface which results from removing from $S_0^n$ a closed regular neighbourhood of each element of $\sigma$. Write $S_0^n - \sigma= Y_1\sqcup \ldots\sqcup Y_k$, observing that each $Y_j$ is a sphere with marked points. 
There is an obvious surjective homomorphism \[(\Mod(S_0^n))_\sigma \to \Mod(\sqcup_i Y_i, \sigma), \] called the {\em cutting homomorphism} associated to $\sigma$. Here, $\Mod(\sqcup_i Y_i, \sigma)$ denotes the subgroup of $\Mod(\sqcup_i Y_i)$ whose elements preserve the set of punctures of $\sqcup_i Y_i$ that correspond to elements $\sigma$.  The cutting homomorphism fits in a short exact sequence

\begin{align}
\label{def-cut-hom}
1 \to T_\sigma \to  (\Mod(S_0^n))_\sigma \to \Mod(\sqcup_i Y_i,\sigma) \to 1
    \end{align}
   where $T_\sigma$ is
the free abelian group  generated by the Dehn twists along the elements of $\sigma$. 

Armed with these definitions, we can give a {\em canonical form} for elements of $\PMod(S_0^n)$. More concretely, let $f \in \PMod(S_0^n)$, and write $\sigma$ for its canonical reduction system, so that $f\in (\Mod(S_0^n))_\sigma$. Again, let  $S_0^n - \sigma= Y_1\sqcup \ldots\sqcup Y_k$. Since $f$ is pure, it follows $f(\alpha) = \alpha$ for every $\alpha \in \sigma$; also, $f(Y_i) = Y_i$ for every $i$. From this discussion, and using the Nielsen-Thurston classification, we have deduced: 

\begin{lem}
With the notation above, the image of $f\in \PMod(S_0^n)$ under the cutting homomorphism \eqref{def-cut-hom} belongs to 
$\PMod(Y_1) \times \cdots \times \PMod(Y_k)$. Moreover, the projection of this image onto each factor is either the identity or pseudo-Anosov.
\label{lem:purecut}
\end{lem}

\subsubsection{Normalizers}
We will use the following well-known result about normalizers of pseudo-Anosov elements: 

\begin{lem}
Let $f\in \Mod(S_0^n)$ be a pseudo-Anosov. Then its normalizer $N_{\Mod(S_0^n)}(f)$ is virtually cyclic.
\label{lem:pA-norm}
\end{lem}

It is also possible to describe the normalizer of a  multitwist. Indeed, observe that, for any $f\in \Mod(S_0^n)$, we have $fT_\sigma f^{-1} = T_{f(\sigma)}$. In particular, we obtain:

\begin{lem}
For any multicurve $\sigma$, $N_{\Mod(S_0^n)} (T_\sigma) = \Mod(S_0^n)_\sigma$.
\label{lem:twistnorm}
\end{lem}

\subsection{Braid groups} 
Given $n\ge 0$, we denote by $F_n$ the {\em configuration space} of $n$ distinct points on a sphere.
Note that the symmetric group $\Sigma_n$ acts on $F_n$ by permutation the coordinates; the quotient space $ J_n=F_n/\Sigma_n$ may then be regarded as the configuration space of $n$ {\em unordered} points on the sphere. Birman \cite[Prop 1.1]{birman} proved that the natural projection $ F_n\to J_n$ is a regular $(n!)$-fold covering map 
 
We define the {\em $n$-strand spherical braid group} as $B_n=\pi_1(J_n)$, and its \emph{pure} subgroup  as $P_n= \pi_1(F_n) < \pi_1(J_n)$.

As mentioned in the introduction,  braid groups are strongly related to mapping class groups of spheres. More concretely, for $n\ge 3$ there is a short exact sequence (see, for instance, \cite[Section 9.4.2]{margalit}): 

\begin{equation}\label{suc-ex-sphere}
1	\to \z_ 2 \to B_n \to \Mod(S_0^n) \to 1
\end{equation}
where $\z_2$ is generated by the  full twist braid, $\Delta_n$, of  $B_n$ and it generates the center of $B_n$.  In turn, for pure braid groups we have:

\begin{equation}\label{suc-ex-sphere-pure}
1	\to \z_ 2 \to P_n \to \PMod(S_0^n) \to 1
\end{equation}

\section{General results on geometric dimension} 
In this section we introduce the main ingredient in our proofs, namely the result of L\"uck-Weiermann \cite{luck-weiermann} stated as Theorem \ref{thm-luck-weierm} below.

\subsection{The main tool}
Let $G$ be a group, and $\C_G$ the family of infinite, virtually-cyclic subgroups of $G$. After \cite{luck-cat(0)-vc} and \cite{luck-weiermann}, we define an  equivalence relation $\sim$ on $\C_G$ by: 

\begin{equation}\label{rel-eq}
C\sim D  \iff |C\cap D|=\infty .
\end{equation}

Let $[\C_G]$ denote the set of equivalence classes and by $[C]$ the equivalence class of  $C\in \C_G$.      
The normalizer of $[C]$ is defined as:
  \begin{equation} \label{def-comm}
   N_{G}[C]:=\{g\in G \mid |gCg^{-1}\cap C|=\infty \};
  \end{equation}
  in other words, it is the commensurator of $C$ in $G$.  We define the following family of subgroups of $N_G[C]$:
   \begin{equation}\label{gh}
   \g_G[C]=\{H\in \vc_{N_G[C]} \mid  |H:H\cap C|<\infty\} \cup
   \fin_{ N_G[C]}.
   \end{equation}

After all these definitions, we are ready to give L\"uck-Weiermann's bound from \cite{luck-weiermann}:    
   
\begin{thm}\cite[Thm. 2.3]{luck-weiermann} \label{thm-luck-weierm}
Let $\C_G$ and $\sim$ be as above. Let $\mathfrak{I}$ be a complete system of representatives, $[H]$, of the $G$-orbits in $[\C_G]$ under the $G$-action coming from conjugation. Suppose there exists $d\in \mathbb{N}$ with the following properties: 
\begin{enumerate}
\item $\gdf G \leq d$,
\item $\gdf N_{G}[H] \leq d-1$, and
\item  $\gd_{\g[H]}N_{G}[H] \leq d$,
\end{enumerate}
  for each $[H]\in \mathfrak{I}$.
Then $\gdvc G\leq d$.  
\end{thm}

Under certain circumstances, Theorem \ref{thm-luck-weierm} becomes a lot easier to work with. In this direction, say that  a group $G$ has {\em property (C)} (for ``conjugation") if, whenever  $f,g\in G$ are elements of infinite order with $gf^mg^{-1}=f^k$, we have that $|m|=|k|$. If $G$ has property (C), then  \cite[Lem. 4.2]{luck-cat(0)-vc} yields that for any $C\in \C_G$,   \[N_G(C)\subseteq N_G(2!C) \subseteq N_G(3!C)\subseteq \cdots, \] where $k!C=\{h^{k!}\mid h\in C \}$ and   $N_G[C]= \cup_{k\geq 1}N_G(k!C)$. 

Finally, say that $G$ has the property of {\em uniqueness of roots} if  for any $f,g\in G$ such that $f^n = g^n$ implies that $f=g$. We have:

\begin{prop} \label{prop:main} Suppose $G$ satisfies property (C) and has a finite index normal subgroup $H$ with the property of uniqueness of  roots.  If for any $C\in \C_H$ we have 
	\begin{enumerate}
		\item[(i)] $\gdf G \leq d$,
		\item[(ii)] $\gdf N_G(C)\leq d-1$, 
	\item[(iii)] $\gdf W_G(C) \leq d$,
	\end{enumerate}
where $W_G(C)=N_G(C)/C $. Then $\gdvc G \leq d$.

\end{prop}
\begin{proof}
We will use Theorem \ref{thm-luck-weierm}. Since $H$ is a normal subgroup of finite index with the property of uniquneness of  roots, we have  $ N_G(D)=N_G(t D)$,
for any $D\in \C_H$ and any $t \in \z\setminus\{0\}$

Let $C\in \C_G$. Combining this with  \cite[Lem. 4.2]{luck-cat(0)-vc},  we have that $N_G[C]=N_G(k!C)$ for some $k\in \z\setminus \{0\}$ and   $k!C\in\C_H $.  Thus we may assume that $C\in \C_H$  and  $N_G[C]=N_G(C)$.  Further, a model for $\efine W_G(C) $ is a model for $E_{\g_G}N_G[C]$ with the action given from the projection $p\colon N_G(C) \to W_G(C)$. 
Applying Theorem \ref{thm-luck-weierm} we conclude the Proposition. 
\end{proof}
 
 \begin{rem}\label{rem-prop-max}
 Let $G$ and $H$ be as in Proposition \ref{prop:main}. Suppose that $H$  satisfies  property (Max).    Note that if $D\in \C_H$  and $D_{max}\in \C_H$ is the  maximal cyclic subgroup containing $D$,  then $N_G(D)=N_G(D_{max})$;  this follows from the property of uniqueness of roots and because $H$ is a normal subgroup of finite index in $G$.  
 Therefore,  in Proposition \ref{prop:main} we may assume that $C\in \C_H$ is maximal in $C_H$.
%
 \end{rem}

\subsection{On proper geometric dimension}
In the light of Proposition \ref{prop:main}, in order to estimate the virtually-cyclic dimension, one needs to be able estimate proper geometric dimension. With this motivation, we now present a number of known results about proper geometric dimension. 

First, an immediate consequence of the definition of proper geometric dimension is that, for any two groups $G_1,G_2$, one has\\
\begin{equation}\label{efin-product}
\gdf(G_1\times G_2)\leq \gdf G_1  + \gdf G_2.
\end{equation}
Another observation is that if $H$ is a subgroup of a group $G$, then 
\begin{equation}\label{efin-subgroup}
\gdf H \le \gdf G. 
\end{equation}

Next, a result of Karrass-Pietrowski-Solitar \cite{KPS} implies that the Bass-Serre tree of a virtually-free group $G$ is a model for $\efin$. In other words, we have:

\begin{lem}\label{lem:efin-vf}
Let $G$ be a virtually-free group. Then $\gdf G \le 1$, with equality if and only if $G$ is infinite. 
\end{lem}

The next theorem, due to L\"uck \cite{luck-type}, gives a relation between the geometric dimension of a group and that of finite-index subgroups: 
  
\begin{thm}\cite[Thm. 2.4]{luck-type}\label{thm:finite-index-gd}
If  $H\subseteq G$ is a subgroup of finite index $n$, then $\gdf G \leq  \gdf H \cdot n\;$ 
\end{thm}

We will also need to be able to bound the proper geometric dimension of certain extensions of groups. In this direction, we will use the next result, which is a consequence of \cite[Thm. 5.16]{luck}:

\begin{thm}\label{thm:seq-efin}
Let $1 \to H \to G \to K\to 1$ be an exact sequence of groups.  Suppose that $H$ has the property that for any group $\tilde{H}$ which contains $H$ as subgroup of finite index, $\gdf \tilde{H}\leq n$. If $\gdf K\leq k$, then $\gdf G\leq n+k$.    
\end{thm}
Finally, we will make use the following well-known result \cite[Prop. 2.6]{MF-BN} in order to prove Corollaries \ref{cor:positgen} and \ref{cor:braids}: 

\begin{lem}
	Suppose $\gdvc G \geq 3$. Let $1 \to F \to G \to H \to 1$ be a short exact sequence of groups, where $F$ is finite. Then $\gdvc G = \gdvc H$. 
	\label{lem:extend}
\end{lem}

\section{Proof of Theorem \ref{thm:main}}

In this section we prove Theorem \ref{thm:main}, using Proposition \ref{prop:main} to $\Mod(S_0^n)$, with $n\le 6$. In order to do so, we first remark that the second and third named authors showed that $\Mod(S_g^n)$ has property (C) \cite{Daniel-Ale}, and that \cite[Theorem 6.1]{bonatti-paris} implies that $\PMod(S_0^n)$ has unique roots; we recall that $\PMod(S_0^n)$ has finite index in $\Mod(S_0^n)$. 
Next, we will use the following special case of the main result of \cite{AMP}, combined with Harer's calculation \cite{Harer} of the virtual cohomological dimension of the mapping class group: 

\begin{thm}
For every $n$, $\gdf \Mod(S_0^n) = n-3$.
\end{thm}

In the light of this result, inequality \eqref{efin-subgroup} implies that 
\[\gdf N_{\Mod(S_0^n)}(f) \leq n-3,\] for any $f\in \Mod(S_0^n)$. We will show in Lemma \ref{lem-max-0}  that $\PMod(S_0^n)$ has the property (Max).  
Therefore, by Remark \ref{rem-prop-max}, the proof of Theorem \ref{thm:main} boils down to proving that \[\gdf W_{\Mod(S_0^n)}(f)\leq  n-3,\] for every infinite-order element $f\in \PMod(S_0^n)$ such that $\langle f \rangle$ is maximal.
We will do so using a case-by-case analysis depending on the Nielsen-Thurston type of such a mapping class.
We have separated the proof in the cases $n=5$ and $n=6$, since the combinatorial possibilities are different in these two cases. 

\begin{rem}
As hinted in Remark \ref{rem:casuistic}, our methods will not carry over to an arbitrary number of punctures. In a nutshell, the reason lies in the difference between the mapping class group of a disconnected surface, and the product of mapping class groups of the components. While the latter is a finite-index subgroup of the former, the index grows with the topology of the surface. For $n\in \{5,6\}$, however, this index is amenable to our computations. 

On the other hand, we stress that essentially the same analysis as in the case $n=6$ will give a direct proof of Corollary \ref{cor:positgen}, as well as the analogous statement for $S_1^3$. 
\label{rem:casuistic2}
\end{rem}

\subsection{The case of the five-punctured sphere} As indicated above, we need to prove  \[\gdf W_{\Mod(S_0^5)}(f)\leq  3\] for every infinite-order element $f\in \PMod(S_0^5)$ such that $\langle f \rangle$ is maximal.  There are two cases to consider: 

\medskip

\noindent{\bf Case 1: $f$ is pseudo-Anosov.} Here, Lemma \ref{lem:pA-norm} implies that  $N_{\Mod(S_0^5)}(f)$ is virtually cyclic, in which case  $\gdf N_{\Mod(S_0^5)}(f) =1$ and   $\gdf W_{\Mod(S_0^5)}(f)=0$.  

\medskip

\noindent{\bf Case 2: $f$ is reducible.} Let $\sigma$ be its canonical reduction system. We distinguish the following further cases, depending on whether $\sigma$ has one or two elements.

\smallskip

{\em Subcase 2(a): $\sigma$ has exactly one element.} Write $\sigma= \{\alpha\}$, observing that $S_0^5 \setminus \alpha = S_0^3 \sqcup S_0^4$. Let $\rho$ be the cutting homomorphism \eqref{def-cut-hom} associated to $\sigma$. Suppose first that $\rho(f)$ is trivial, so that $f\in \langle T_\alpha \rangle$. By Lemma \ref{lem:twistnorm}, $N_{\Mod(S_{0,5})} (T_\alpha) = \Mod(S_{0,5})_\alpha$, and thus we have:  

\begin{equation}
1 \longrightarrow \langle T_\alpha \rangle \longrightarrow N_{\Mod(S_0^5)}(f) \longrightarrow \Mod(S_0^3,q_1) \times \Mod(S_0^4,q_2) \longrightarrow 1
\end{equation}
where the punctures $q_1,q_2$ are those  that appear  when the surface is cut along  $\alpha$ (see Figure \ref{fig-cutting-c}).
\begin{figure}
\begin{center}
\includegraphics[scale=0.2]{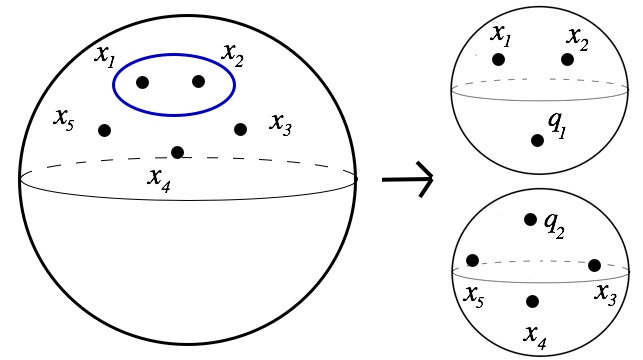}
\caption{Cutting along a curve.}
\label{fig-cutting-c}
\end{center}
\end{figure}
Therefore
\begin{equation}
 W_{\Mod(S_0^5)}(f)\simeq \Mod(S_0^3,q_1) \times \Mod(S_0^4,q_2)
 \end{equation}  
Since $\Mod(S_0^3)$ is finite and $\Mod(S_0^4)$ is virtually free, the combination of Lemma \ref{lem:efin-vf},  with equations \eqref{efin-product} and    \eqref{efin-subgroup} implies that  $\gdf W_{\Mod(S_0^5)}(f) =  1 $, as desired.

Suppose now that $\rho(f)$ is not trivial, so that the restriction of $f$  to $ \Mod(S_0^4,q_2)$ (using the notation above) is a pseudo-Anosov, which we denote $f_2$. In this case, we have: 

\begin{equation}
1 \longrightarrow \langle T_{\alpha}\rangle \to N_{\Mod(S_0^5)}(f) \longrightarrow \Mod(S_0^3,q_1) \times N_{\Mod(S_0^4,q_2)}(f_2)  \longrightarrow 1
\end{equation}

By Lemma \ref{lem:pA-norm}, $N_{\Mod(S_0^4,q_2)}(f_2)$ is virtually cyclic, hence taking quotients we obtain: 

\begin{equation}
1 \to \mathbb{Z} \to  W_{\Mod(S_0^5)}( f ) \to F \to 1,
\end{equation}
where $F$ is a finite group. In other words, $W_{\Mod(S_0^5)}( f )$ is virtually cyclic, and thus $\gdf W_{\Mod(S_0^5)}(f )= 1$ by Lemma \ref{lem:efin-vf}.

\medskip

{\em Subcase 2(b): $\sigma$ has two elements.} Write $\sigma= \{\alpha,\beta\}$, and $S \setminus \sigma = Y_1 \sqcup Y_2 \sqcup Y_3$. Note that $Y_j$ is homeomorphic to $S_0^3$, for $j=1,2,3$. 

Since $f$ is pure, it follows that $f\in \langle T_\alpha, T_\beta\rangle$. Therefore, the normalizer of $f$ in $\Mod(S_0^5)$ coincides with $\Mod(S_0^5)_\sigma$, by Lemma \ref{lem:twistnorm}. The cutting homomorphism \eqref{def-cut-hom} reads 
\begin{equation}
1\to \longrightarrow \langle T_\alpha, T_\beta\rangle \longrightarrow \Mod(S_0^5)_\sigma \longrightarrow \Mod( Y_1 \sqcup Y_2 \sqcup Y_3,\sigma) \longrightarrow 1.
\end{equation} 
Since $\Mod( Y_1 \sqcup Y_2 \sqcup Y_3)$ is a finite group, we obtain \[1 \longrightarrow \mathbb{Z} \longrightarrow  W_{\Mod(S_0^5)}( f ) \longrightarrow F',\]
where $F'$ is a finite group. Therefore, $\gdf W_{\Mod(S_0^5)}(f )= 1$, again by Lemma \ref{lem:efin-vf}.
This finishes the proof of Theorem \ref{thm:main} in the case $n=5$. 

\subsection{The case of the six-punctured sphere} We now prove Theorem \ref{thm:main} in the case $n=6$. Again, it suffices to prove that \[\gdf W_{\Mod(S_0^5)}(f) \leq 4\] for every $f \in \PMod(S_0^6)$ of infinite order such that $\langle f \rangle$ is maximal.  
Let $f$ be such an element. As in the case $n=5$, if $f$ is pseudo-Anosov, then  $\gdf N_{\Mod(S_0^6)}(f) =1$ and  $\gdf W_{\Mod(S_0^6)}(f)=0$. Therefore, from now on we assume that $f$ is reducible. Write $\sigma$ for the canonical reduction system of $f$, noting that $1\le |\sigma| \le 3$. 

\bigskip

\noindent{\bf Case 1: $|\sigma|=1$.} We write $\sigma=\{\alpha_1\}$, and distinguish the following subcases: 

\medskip

{\em Subcase 1(i): $\alpha_1$ bounds a disk with exactly two punctures.} In this case, $S_0^6 \setminus \alpha_1 = S_0^3 \sqcup S_0^5$, and  the cutting homomorphism \eqref{def-cut-hom} associated to $\sigma$ reads: 

\begin{equation}\label{eq:six0} 1
\to \langle T_{\alpha_1} \rangle \to \Mod(S_0^6)_\sigma \to \Mod(S_0^3 \sqcup S_0^5, \sigma) \to 1
\end{equation} 
Since the two components of $S_0^6 \setminus \alpha_1$ are not homeomorphic (or by Lemma \ref{lem:purecut}) we deduce that $\rhos(\Mod(S_0^6)_{\sigma})= \Mod(S_0^3,q_1)\times \Mod(S_0^5,q_2)$, where $q_1$ and $q_2$ are the punctures that appear  when cutting $S_0^6$ along  $\alpha_1$. Restricting this sequence to $N_{\Mod(S_0^6)}(f)$, and observing that $\Mod(S_0^3,q_1) \cong \mathbb{Z}_2$, we obtain: 

\begin{equation}\label{eq:six1}
1\to \langle T_{\alpha_1} \rangle \to N_{\Mod(S_0^6)}(f) \to \mathbb{Z}_2 \times \Mod(S_0^5,q_2) \to 1 
\end{equation}

Suppose first that $f$ has no pseudo-Anosov components; in other words, the projection of $f$ under the cutting homomorphism is trivial. In this case, $f$ is central in $N_{\Mod(S_0^6)}(f)$, and from \eqref{eq:six1} we obtain \[W_{\Mod(S_0^6)}(f) \cong \mathbb{Z}_2 \times \Mod(S_0^5, q_2).\] Note that any model for $\efin(\Mod(S_0^5))$ is also a model for $\efin(\mathbb{Z}\times \Mod(S_0^5))$ also; thus $\gdf W_{\Mod(S_0^6)}(f) \le \gdf \Mod(S_0^5) \le 2$.

Thus, we may assume that the restriction of $f$ to the $S_0^5$-component of $S_0^6 - \alpha_1$ is pseudo-Anosov. In this case, Lemma \ref{lem:pA-norm} and  equation \eqref{eq:six0} yield:

\begin{equation}
1\to \z \to N_{\Mod(S_0^6)}(f) \to \z_2 \times V\to 1,  
\end{equation}
where $V\subseteq N_{\Mod(S_0^5)}(f_2)$ is infinite and virtually cyclic.  Thus, taking quotients:
\begin{equation}
1\to \z  \to  W_{\Mod(S_0^6)}(f) \to  \z_2 \times F\to 1,  
\end{equation}
and hence $\gdf W_{\Mod(S_0^6)}(f)\leq 1$, as desired. 

\medskip

{\em Subcase 1(ii): Each component of $S_0^6 \setminus \alpha_1$ contains three punctures.} In this case, $S_0^6 \setminus \alpha_1 = S_0^4 \sqcup S_0^4$, and thus  \[\Mod(S_0^4\sqcup S_0^4)\stackrel{\psi}{\simeq} (\Mod(S_0^4)\times \Mod(S_0^4))\rtimes \z_2,\] where $\z_2$ is generated by a mapping class that interchanges the two components of $S_0^6 \setminus \alpha_1$. Furthermore, the image of $\Mod(S_0^6)_\sigma$ under the cutting homomorphism \eqref{def-cut-hom} is equal to \[\rhos (\Mod(S_0^6)_{\sigma})\stackrel{\psi\rhos}\simeq (\Mod(S_0^4,q_1)\times \Mod(S_{0}^4,q_2))\rtimes \z_2\] where $q_1$ and $q_2$ are again the new punctures of $S_0^6 \setminus \alpha_1$.  
Let $\Mod(S_0^6)^*_{\sigma}\subseteq \Mod(S_0^6)_{\sigma}$ be the subgroup whose elements do not permute the components of $S_0^6 \setminus \alpha_1$, and let $\rhos^*:=\rhos|_{\Mod(S_0^6)^*_{\sigma}}$. We have the following diagram:

 \begin{equation}\label{eq-1-stab}
 \xymatrix{
 &&& 1 \ar[d]&
 \\
 1\ar[r]& \langle T_{\alpha_1} \rangle \ar[d]^{Id} \ar[r]& \Mod(S_0^6)_{\sigma}^* \ar[d]^{inclusion}\ar[r]^(0.3){\psi\rhos^*}& \Mod(S_0^4,q_1)\times \Mod(S_0^4,q_2)\ar[r] \ar[d]& 1 
 \\
 1\ar[r]& \langle T_{\alpha_1} \rangle  \ar[r]& \Mod(S_0^6)_{\sigma} \ar[r]^(0.3){\psi\rhos}& (\Mod(S_0^4,q_1)\times \Mod(S_0^4,q_2))\rtimes \z_2 \ar[d]\ar[r]& 1 
 \\
 &&& \z_2\ar[d]&
 \\&&& 1 &
 }
 \end{equation}

With the above diagram in mind,  we distinguish the following two cases, depending on the image of $f$ under the cutting homomorphism associated to $\sigma$: 

\medskip

\noindent (a) Suppose first that $f$ has no pseudo-Anosov components. In this case, $N_{\Mod(S_0^6)}(f) = \Mod(S_0^6)_\sigma$. Since $\Mod(S_0^4)$ is virtually free, equation \eqref{efin-product} and Lemma \ref{lem:efin-vf} imply that \[\gdf\left (\Mod(S_0^4,q_1)\times \Mod(S_0^4,q_2)\right )\le 2,\] which in turn yields \[\gdf \left (\Mod(S_0^4,q_1)\times \Mod(S_0^4,q_2))\rtimes \z_2\right )\le 4,\] by Theorem  \ref{thm:finite-index-gd}. Finally, using Theorems \ref{thm:seq-efin} and Lemma \ref{lem:efin-vf}, we obtain $\gdf N_{\Mod(S_0^6)}(f)\leq 5$ and $ \gdf W_{\Mod(S_0^6)}(f)\leq 4$, as desired. 

\smallskip

\noindent (b) Now suppose that $f$ has at least one pseudo-Anosov component; equivalently, assume that$\psi\rhos(f)=(f_1,f_2,Id_{\z_2})$ is not trivial. Again, there are two cases two consider. 

Suppose first that there exists $(g_1,g_2,\gamma)$  in the image $\psi\rhos(N_{\Mod(S_0^6)}(f))$ with $\gamma\neq Id_{\z_2}$. In particular, $f_1$ is conjugate to $f_2^{\pm 1}$, and hence both $f_1$ and $f_2$ are pseudo-Anosov. Let \[N_{\Mod(S_0^6)}(f)^*=N_{\Mod(S_0^6)}(f) \cap \Mod(S_0^6)_{\sigma}^*.\] By restricting the diagram \eqref{eq-1-stab} we have
\begin{align}\label{n-one-four}
 \xymatrix{
 &&
 & 1 \ar[d]&
 \\
 1\ar[r]
 & \langle T_{\alpha_1} \rangle \ar[d]^{Id} \ar[r]
 & N_{\Mod(S_0^6)}(f)^* \ar[d]^{inclusion}\ar[r]^(0.5){\psi\rhos^*}
 & V_1\times V_2 \ar[r] \ar[d]
 & 1 
 \\
 1\ar[r]
 & \langle T_{\alpha_1} \rangle  \ar[r]
 & N_{\Mod(S_0^6)}(f) \ar[r]^(0.4){\psi\rhos}
 & (V_1\times V_2 )\rtimes \z_2 \ar[d]\ar[r]
 & 1 
 \\
 &&
 & \z_2\ar[d]&
 \\&&
 & 1 &
 }
 \end{align}
  where $V_1\times V_2\subseteq N_{\Mod(S_0^4,q_1)}(f_1)\times N_{\Mod(S_0^4,q_2)}(f_2)$, which is a product of virtually cyclic subgroups by Lemma \ref{lem:pA-norm}. 
Taking quotients in \eqref{n-one-four} we obtain: 
\begin{align}\label{w-one-four}
 \xymatrix{
1\ar[r]
 & \z   \ar[r]
 & W_{\Mod(S_0^6)}(f) \ar[r]
 & V_3,
  \ar[r]
 & 1 
}
 \end{align}   where $V_3$ is a virtually cyclic subgroup.
In particular, this implies that
 \begin{align}
 \gdf W_{\Mod(S_0^6)}(f)\leq 2,
 \end{align}
using Theorem \ref{thm:seq-efin}.  This finishes the proof of the case in consideration.

\smallskip

Next, suppose that for any element $g\in N_{\Mod(S_0^6)}(f)$, $\psi \rhos(g)=(g_1,g_2,Id_{\z_2})$, and thus $N_{\Mod(S_0^6)}(f)=N_{\Mod(S_0^6)}(f)^*$. Hence   $N_{\Mod(S_0^6)}(f)$ and $W_{\Mod(S_0^6)}(f)$ fit into the  short exact sequence, 
\begin{align}\label{eq-1}
 \xymatrix{
 1\ar[r]
 & \z  \ar[r]
 & N_{\Mod(S_0^6)}(f) \ar[r]^(0.5){\psi\rhos|}
 & V_1\times \Mod(S_0^4,q_2) \ar[r] 
 & 1 
 \\
 1\ar[r]
 & \z  \ar[r]
 & W_{\Mod(S_0^6)}(f) \ar[r]
 & F \times \Mod(S_0^4,q_2) \ar[r] 
 & 1, }
 \end{align} 
 in the case when $f_1$ is pseudo-Anosov and $f_2$ is the identity, or into
 \begin{align}\label{eq-2}
 \xymatrix{
 1\ar[r]
 & \z 
  \ar[r]
 & N_{\Mod(S_0^6)}(f) \ar[r]^(0.5){\psi\rhos}
 & V_1\times V_2 \ar[r] 
 & 1 
 \\ 
 1\ar[r]
 & \z 
  \ar[r]
 & W_{\Mod(S_0^6)}(f) \ar[r]
 & V_3 \ar[r] 
 & 1
 }
\end{align} when  or both $f_1$ and $f_2$ are pseudo-Anosov; here,
 $V_1$, $V_2$ and $V_3$ are virtually cyclic subgroups and $F$ is finite.
Proceeding as above, in both cases  (\ref{eq-1}) and (\ref{eq-2}) we conclude that $\gdf W_{\Mod(S_0^6)}(f)\leq 2$. 
This finishes the discussion of Case 1.

\bigskip

\noindent{\bf{Case 2. $|\sigma| =2$.}} Write $\sigma=\{\alpha_1,\alpha_2\}$. Again, there are some cases to consider, depending on the topological type of $\alpha_1$ and $\alpha_2$. 

\smallskip

{\em Subcase 2(i):   $\alpha_i$ bounds a disc with exactly two punctures, for $i=1,2$.}
Observe that $S_0^6 \setminus (\alpha_1 \cup \alpha_2) = S_0^4 \sqcup S_0^3 \sqcup S_0^3$. The cutting homomorphism \eqref{def-cut-hom}  yields the exact sequence $$1\to  \langle T_{\alpha_1},T_{\alpha_2}\rangle \to \Mod(S_0^6)_{\sigma}\to \Mod(S_0^4\sqcup S_0^3\sqcup S_0^3,\sigma)\to 1,$$ noting that  $\langle T_{\alpha_1},T_{\alpha_2}\rangle \simeq \z^2$. Observe that  
\begin{align}
\Mod(S_0^4\sqcup S_0^3\sqcup S_0^3,\sigma) & \stackrel{\psi}{\simeq} \Mod(S_0^4,q_1, q_2)\times (\Mod(S_0^3,q_3)\times \Mod(S_0^3,q_4))\rtimes \z_2 \\ & \stackrel{\phi} \simeq \Mod(S_0^4,q_1,q_2)\times (\z_2\times \z_2)\rtimes \z_2
\end{align}
Again, there are different cases depending on the image of $f$ under the cutting homomorphism. In this direction, suppose first that 
$f= T_{\alpha_1}^{k_1}T_{\alpha_2}^{k_2}$ with $\gcd(k_1,k_2)=1$. In this case $N_{\Mod(S_0^6)}(f)=\Mod(S_0^6)_{\sigma}$, and we have:
\begin{equation}\label{eq-3}
 1\to \z^2  \to  N_{\Mod(S_0^6)}(f) \xrightarrow{\phi\psi\rhos}
  \Mod(S_0^4,q_1,q_2)\times (\z_2\times \z_2)\rtimes \z_2 \to
 1 
 \end{equation}
 \begin{equation}
 1\to
  \z  \to
 W_{\Mod(S_0^6)}(f) \to \Mod(S_0^4,q_1,q_2)\times (\z_2\times \z_2)\rtimes \z_2  \to 1
 \label{eq-3-bis}
 \end{equation}
From the exact sequences \eqref{eq-3} and \eqref{eq-3-bis}, plus (\ref{efin-product}) and Theorem \ref{thm:seq-efin},   we conclude that $\gdf W_{\Mod(S_0^6)}(f)\leq 2$, as desired. 

Suppose now that  $\psi\rhos(f)=(f_1,Id)\in \Mod(S_0^4,q_1,q_2)\times F$, with $f_1$ is pseudo-Anosov. Then 
\[ 1\to
  \z^2  \to
 N_{\Mod(S_0^6)}(f) \xrightarrow{\phi\psi\rhos}
 V\times F \to
 1 \]
 and 
 \[
 1\to
  \z^2  \to
  W_{\Mod(S_0^6)}(f) \xrightarrow{\psi\rhos}
  F\times F \to 
  1
 \]
where $F, F'$ are finite groups and $V$ is virtually cyclic. By Theorem \ref{thm:seq-efin} we conclude  
\[
\gdf N_{\Mod(S_0^6)}(f)\leq 3 \;\;\text{and}\;\;
\gdf N_{\Mod(S_0^6)}(f)\leq 2,
\]
as desired.

{\em Subcase 2(ii):   $\alpha_1$ bounds a disc with exactly two punctures, and $\alpha_2$ bounds a disc with three punctures. } In this case,  the cutting homomorphism \eqref{def-cut-hom} again gives:
\[1\to  \langle T_{\alpha_1},T_{\alpha_2}\rangle \to \Mod(S_0^6)_{\sigma}\to \Mod(S_0^4\sqcup S_0^3\sqcup S_0^3,\sigma)\to 1.\] However, in this case we have: 
\begin{align*}
\psi \rhos(\Mod(S_0^6)_{\sigma})&= \Mod(S_0^3,q_1)\times \Mod(S_0^3,q_2,q_3)\times \Mod(S_0^4,q_4)\\
&\stackrel{\nu}{\simeq} \z_2 \times \Mod(S_0^4,q_4).  \end{align*}
Again, we distinguish two cases depending on the image of $f$ under the cutting homomorphism. First, assume that $f=T_{\alpha_1}^{k_1}T_{\alpha_2}^{k_2}$ with $\gcd(k_1,k_2)=1$. Then  $N_{\Mod(S_0^6)}(f)=\Mod(S_0^6)_\sigma$, and therefore 
we have the sequences 
\[
 1\to
  \z^2  \to
 N_{\Mod(S_0^6)}(f) \xrightarrow{\nu\rhos}
   \z_2\times \Mod(S_0^4,q_4) \to
 1
\]
 and
\[
 1\to
\to \z  \to
  W_{\Mod(S_0^6)}(f) \xrightarrow{\nu\rhos}
   \z_2\times \Mod(S_0^4,q_4) \to
 1.
\]
From these sequences, we conclude that
\begin{align*}
\gdf N_{\Mod(S_0^6)}(f)\leq  3 \;\; \text{and}\;\; \gdf W_{\Mod(S_0^6)}(f) \leq 2,
\end{align*} 
as desired. 
\smallskip

Suppose now that $\nu\rhos f =(Id_{\z_2}, f_1)$, where $f_1$ is pseudo-Anosov. From Lemma \ref{lem:pA-norm}  we have  the sequences
\[
 1\to 
 \z^2 \to 
  N_{\Mod(S_0^6)}(f) \xrightarrow{\nu\rhos}
   V\to
 1\]
and
\[
 1\to 
 \z^2  \to
  W_{\Mod(S_0^6)}(f) \xrightarrow{\nu\rhos}
  F\to
 1,
\]
where $V'$ is a virtually cyclic subgroup and $F$ is finite. Therefore $\gdf W_{\Mod(S_0^6)}(f) \leq 2$ again. This finishes the discussion of Case 2. 

\bigskip

\noindent{\bf{Case 3. $|\sigma| =3$.}} Write $\sigma=\{\alpha_1,\alpha_2,\alpha_3\}$, observing that $S_0^6 \setminus (\alpha_1\cup \alpha_2\cup \alpha_3)$ is the disjoint union of four copies of  $S_0^3$. Thus the cutting homomorphism \eqref{def-cut-hom} gives
\[1\to  \langle T_{\alpha_1},T_{\alpha_2}, T_{\alpha_3}\rangle \to \Mod(S_0^6)_{\sigma}\xrightarrow{\rho_\sigma} \Mod(S_0^3\sqcup S_0^3\sqcup S_0^3 \sqcup S_0^3,\sigma)\to 1, \] observing that  $\langle T_{\alpha_1},T_{\alpha_2}, T_{\alpha_3} \simeq \z^3$. 
Note that  $\Mod(S_0^3 \sqcup S_0^3 \sqcup S_0^3 \sqcup S_0^3)$ is a finite subgroup,  and that $f$ is in the kernel of $\rhos$; moreover, $f=T_{\alpha_1}^{k_1}T_{\alpha_2}^{k_2}T_{\alpha}^{k_3}$  with $\gcd(k_1,k_2,k_3)=1$. Therefore we have the sequences
\[
1\to \z^3 \to 
 N_{\Mod(S_0^6)}(f) \to
  F \to 1
\] and 
\[ 1\to  \z^2 \to  W_{\Mod(S_0^6)}(f) \to  F \to 1.\]
 In particular, $\gdf W_{\Mod(S_0^6)}(f) \leq 2$, as desired. 
This finishes the discussion of Case 3, and also the proof of Theorem \ref{thm:main}. \qed

\subsection{Proof of Corollaries \ref{cor:positgen} and \ref{cor:braids}} We now explain how to prove Corollaries  \ref{cor:positgen} and \ref{cor:braids}. First, the latter follows immediately from the combination of Theorem \ref{thm:main}, equation \eqref{suc-ex-sphere}, and Lemma \ref{lem:extend}. 

Now, Corollary \ref{cor:positgen} follows along equal lines, recalling that there are short exact sequences 
\[1\to \z_2 \to \Mod(S_1^2) \to \Mod(S_0^5) \to 1\]
and 
\[1\to \z_2 \to \Mod(S_2^0) \to \Mod(S_0^6) \to 1;\] in both cases, the $\z_2$ is generated by a {\em hyperelliptic involution}, 
see \cite{margalit}.

\section{A general bound} 
\label{sec:general}

In this section we prove Theorem \ref{propo:bound}. Before doing so, we remark that Degrijse-Petrosyan \cite{DP} have recently given the following bound for the virtually cyclic dimension of $\Mod(S_g^n)$: 

\begin{thm}[\cite{DP}]
Let $g,n\ge 0$ with $3g-3+n \ge 1$. Then 
$$
\gdvc \Mod(S_g^n) \le 9g+ 3n-8.
$$
\label{prop:DP}
\end{thm}

The above  result is stated in \cite{DP} for closed surfaces only; however the argument remains valid in full generality. For completeness we include a sketch here, which uses known facts about the geometry of the Weil-Petersson metric on Teichm\"uller space. We refer the reader to \cite{Wolpert} for a thorough discussion on these and many other topics. 

\begin{proof}[Proof of Theorem \ref{prop:DP}]
Denote by $T_{g,n}$ the Teichm\"uller space of $S_g^n$, which is homeomorphic to $\mathbb{R}^{6g+2n-6}$.
Endow $T_{g,n}$ with its Weil-Petersson metric, on which $\Mod(S_g^n)$ acts by semisimple isometries. 
 The metric completion $\bar{T}_{g,n}$ of $T_{g,n}$ is a complete separable ${\rm CAT}(0)$ space, and the action of $\Mod(S_g^n)$ on $T_{g,n}$  extends to a semisimple isometric action on $\bar{T}_{g,n}$. Moreover, the stabiliser of a point is a virtually abelian group of rank $\le 3g+n-3$. At this point, \cite[Corollary 3(iii)]{DP} implies that \[\gdvc \Mod(S_g^n) = (6g+2n-6) + (3g+n-3) + 1 = 9g+3n - 8,\] as desired
\end{proof}

We now proceed to prove Theorem \ref{propo:bound}. Again, the main tool will be Proposition \ref{prop:main}, this time combined with a result of Mart\'inez-P\'erez \cite{Conchita}.  Before stating the latter, we need the following definition. 
Let $G$ be a group, and $F\in \fin_G$ a finite subgroup. The {\em length} $l(F)$ of $F$ is defined as the largest natural number $k$ for which there is a chain $1=F_0<F_1 < \cdots < F_k =F$.
The {\em length} of $G$ is 
$$l(G)= \sup \{l(F) \mid F\in \fin_G \}.$$
\begin{thm}\cite[Thm. 3.10, Lem. 3.9]{Conchita} \label{thm-CM}
 Suposse that $ 3\leq \gdf G < \infty $. If $l(G)$ is finite, then    
\begin{align*}
\gdf G \leq \vcd G+l(G),
\end{align*}  
where $\vcd(\cdot)$ denotes virtual cohomological dimension. 
\end{thm}

We begin with the following Lemma:

\begin{lem}\label{lem-gdfwp}
Suppose $n\geq 4$. Let $f\in \PMod(S_0^n)$ and suppose that $\langle f\rangle $ is maximal in $\C_{\PMod(S_0^n)}$. Then 
$$\gdf W_{\PMod(S_0^n)}(f)\leq n-4.$$ 
\end{lem}
\begin{proof}
Suppose that $f$ has canonical reduction system $\sigma=(\alpha_1,...,\alpha_k)$, the restriction  $f|_{Y_j}$ is pseudo-Anosov for $j\in{1,...,r}$ and $f|_{Y_i}$ is the identity for $i\in \{r+1,...,k+1\}$. Note that for $j\in \{1,...,r\}$,  $Y_j$ is a sphere with at least four punctures.   
Thus from the definition of the cutting homomorphism \eqref{def-cut-hom} and the comment after it,   we have
\begin{equation}
1\to  \langle{T_{\sigma}}\rangle \to  N_{\PMod(S_0^n)}(f)  \xrightarrow{\rhos}   U \times\Pi_{i=r+1}^{k+1} \PMod(Y_i)  \to  1 
\end{equation}
where $U$ is a finite-index subgroup of $\Pi_{i=1}^r N_{\PMod(Y_i)}(f|_{Y_i})\simeq \z^r$.  
  If $f\in \langle{T_{\sigma}}\rangle$, then $r=0$ and 
\begin{align}\xymatrix{
1\ar[r] & \langle{T_{\sigma}}\rangle/\langle f\rangle \ar[r]& W_{\PMod(S_0^n)}(f)  \ar[r]^{\widehat{\rhos}}  &  \Pi_{i=1}^{k+1} \PMod(Y_i)  \ar[r]    & 1 
 }
\end{align}
where $\langle{T_{\sigma}}\rangle/\langle f\rangle \simeq \z^{k-1}$. By Theorem \ref{thm:seq-efin}, equation \eqref{efin-product}, and \cite[Corollary 10.5]{HOP} we have 
\begin{align*} 
\gdf W_{\PMod(S_0^n)}(f) &\leq k-1 + \sum_{i=1}^{k+1} \gdf \PMod(Y_i) \\
&=k-1 +\sum_{i=1}^{k+1} \vcd \PMod(Y_i)\\
&= k-1 + n-k-3\\
 &= n-4.
\end{align*}
If,  on the other hand, $f\not\in \langle{T_{\sigma}}\rangle$, then $r\geq 1$.  Let $\bar{f}=(f|_{Y_1},...,f|_{Y_r})$, then  
\begin{equation} \label{seq-wp}
1\to   \langle{T_{\sigma}}\rangle \to  W_{\PMod(S_0^n)}(f) \xrightarrow{\rhos}  U/\langle \bar{f}\rangle \times\Pi_{i=r+1}^{k+1} \PMod(Y_i)  \to 1 
\end{equation}
Note that $U/\langle \bar{f}\rangle$ is virtually $\z^{r-1}$, and thus $\gdf (U/\langle \bar{f}\rangle)= r$. Again, applying Theorem \ref{thm:seq-efin},  equation \eqref{efin-product}, and \cite[Corollary 10.5]{HOP} to (\ref{seq-wp}) we have
\begin{align*}
\gdf W_{\PMod(S_0^n)}(f) &\leq k+1+ r-1 + \sum_{i=r+1}^{k+1} \vcd \PMod (Y_i) 
\\
& \leq k+r-1 + \sum_{i=1}^{k+1} \vcd \PMod (Y_i)-r
\\
&=n-4.
\end{align*}
\end{proof}

We are finally in a position for proving Theorem \ref{propo:bound}: 

\begin{proof}[Proof of Theorem \ref{propo:bound}]
We will  use Proposition \ref{prop:main}. Let  $\langle f \rangle $ maximal in $\C_{\PMod(S_0^n)}$, and 
note  that $\gdf N_{\Mod(S_0^n)}(f)\leq \gdf \Mod(S_0^n)= n-3$, by \cite{AMP}. 
We now give a bound for 
$\gdf W_{\Mod(S_0^n)}(f).$ We have the following exact sequence 
\begin{align*}\xymatrix{
 1\ar[r] & N_{\PMod(S_0^n)}(f)  \ar[r]  &  N_{\Mod(S_0^n)}(f)  \ar[r]  & A \ar[r]  & 1
 } 
\end{align*}
where $A\subseteq \Sigma_n$. Since $\langle f\rangle \subset \PMod(S_0^n)$ we have
\begin{align*}\xymatrix{
1\ar[r] &  W_{\PMod(S_0^n)}(f)  \ar[r] & W_{\Mod(S_0^n)}(f)    \ar[r] &A   \ar[r] & 1
}
\end{align*}  
  We remark that $W_{\PMod(S_0^n)}(f)$ is  torsion-free  since $\PMod(S_0^n)$ is, and $\langle f\rangle$ is maximal in $\C_{\PMod(S_0^n)}$.   Then  $l(W_{\Mod(S_0^n)}(f))\leq l(A)\leq l(\Sigma_n) $. By a result of Cameron-Solomon-Turull \cite[Theorem 1]{Cameron}, \[l(\Sigma_n) =  [\tfrac{3n-1}{2}]-b_n.\] At this point,
 Theorem \ref{thm-CM} and Lemma \ref{lem-gdfwp} together imply   
\begin{align*}
\gdf W_{\Mod(S_0^n)}(f)&\leq  \vcd W_{\Mod(S_0^n)}(f) +l(A) \\ &=  \vcd W_{\PMod(S_0^n)}(f) +l(A) 
\\
&\leq n-4 + [\tfrac{3n-1}{2}]-b_n, 
\end{align*}
where the equality holds since $\PMod(S_0^n)$ has finite index in $\Mod(S_0^n)$. 
\end{proof}

\subsection{Pure subgroups} As mentioned in the introduction, if one considers the pure mapping class group instead of the full mapping class group, then the situation is a lot easier. Indeed, after Harer's calculation of the virtual cohomological dimension of the (pure) mapping class group, we get:

\begin{prop}\label{prop-gdvc-mcgn-1}
	Let $n\geq 4$. Then  \[\gdvc \PMod(S_0^n)= \gdf(\PMod(S_0^n)) + 1= n-2.\]
\end{prop}
The main ingredient of the proof is the following result. After \cite{luck-weiermann}, we say that  a group $G$ satisfies (Max) if every subgroup $H\in \vci_G$ is contained in a unique $H_{max}\in \vci_G$ which is maximal in $\vci_G$. 

\begin{lem} \label{lem-max-0}
	Let $n\ge 4$. Then the pure mapping class group $\PMod(S_0^n)$ satisfies property (Max).
\end{lem}
\begin{proof}
Let $C\subset D$ be an inclusion of infinite  cyclic subgroups. Then the centralizers of $C$ and $D$ in $\PMod(S_0^n)$ are equal, since $\PMod(S_0^n)$ has unique roots.  If $C$ is generated by a pseudo-Anosov class, then  its centralizer is cyclic, and in particular  is the unique  maximal cyclic subgroup that contains $C$.  If $C$ is  generated by a reducible element $f$, by Lemma \ref{lem:purecut} and the case of pseudo-Anosov classes,  we obtain a unique maximal cyclic subgroup that contains $C$. 
\end{proof}

\begin{proof}[Proof of Proposition \ref{prop-gdvc-mcgn-1}] As mentioned in the introduction, L\"uck and Weiermann \cite[Theorem 5.8]{luck-weiermann} proved that every group with property (Max) satisfies inequality \eqref{eq:main}. Now, a combination of Harer's calculation \cite{Harer} of the virtual cohomological dimension of $\Mod(S_0^n)$  and \cite[Corollary 10.5]{HOP}  yields $\gdf \PMod(S_0^n) = n-3$. Therefore, \[\gdvc \PMod(S_0^n) \le n-2.\]
	
	Now, $S_0^n$ contains $n-2$ disjoint essential curves $\alpha_1,...,\alpha_{n-3}$,  and the subgroup $\langle T_{\alpha_1},...,T_{\alpha_{n-3}} \rangle$ is isomorphic to $\z^{n-3}$. By  property (\ref{efin-subgroup}), we conclude  $\gdvc \Mod(S_0^n)\geq \gdvc \z^{n-3}=n-2$. 
\end{proof}

%
%
%

\section{Appendix. Surfaces with boundary} 
Finally, we explain how to establish inequality \eqref{eq:main} in the case of mapping class groups of surfaces with non-empty boundary. As indicated in the introduction, the arguments appear in the paper of Flores and Gonz\'alez-Meneses \cite{FG} in the case when the surface has genus zero. For completeness, we give a self-contained argument here. 

For $S$  a surface with non-empty boundary, its mapping class group $\Mod(S)$ is again defined as the group of isotopy classes of self-homeomorphisms of $S$, but this time the homeomorphisms and isotopies are required to fix each boundary component pointwise. As a by-product of this definition, $\Mod(S)$ has no torsion. 

The main ingredient will be the following result of Mart\'inez-P\'erez \cite{Conchita}; again, $\vcd(\cdot)$ denotes  virtual cohomological dimension: 

\begin{thm}\label{thm-rk-CM}
	Let $G$ be a group such that any finite  subgroup is nilpotent. Suppose  $\vcd G <\infty$ and $\gdf G \geq 3 $, then 
	$$\gdf G \leq \max_{F\in \fin_G} \{\vcd G + rk(W_G F) \},$$
	where  $rk(\cdot)$ denotes the biggest rank of a finite  elementary abelian subgroup. 
\end{thm}

Denoting by $\S_{g,b}^n$ the connected, orientable surface of genus $g$, with $n$ marked points and $b$ boundary components. We remark that for $m\geq 3$ the congruence subgroups $\Mod(S_{g,b}^n)[m]$ are finite-index subgoups that have property (Max) \cite[Prop. 5.11]{Daniel-Ale},  and property of uniqueness of roots \cite{bonatti-paris}. 

We will make use of the following lemma:

\begin{lem}\label{lem-gdf-w-b}
Let $b\geq 1$. If $g=0$,  suppose $b+n\geq 4$, and if  $g\geq 1$,  suppose $2g+b+n\ge 3$.    Fix $m\geq 3$  and let $C\in \C_{\Mod(S_{g,b}^n)[m]}$ maximal,  then 
\begin{align}
\gdf W_{\Mod(S_{g,b}^n)}(C)\leq \gdf \Mod(S_{g,b}^n) +1 .
\end{align}
\begin{proof}
We will use Theorem \ref{thm-rk-CM}. First, the hypotheses imply that $\Mod(S_{g,b}^n$ contains $\mathbb Z^k$ with $k\ge 3$, and thus $\vcd \Mod(S_{g,b}^n) +1 \geq 3$. Also,   observe that $\gdf \Mod(S_{g,b}^n) = \vcd  \Mod(S_{g,b}^n)$ since $\Mod(S_{g,b}^n)$ has no torsion. 

Let $C\in \C_{\Mod(S_{g,b}^n)[m]}$ be maximal.  Note that any finite subgroup of $W_{\Mod(S_{g,b}^n)}(C)$ is of the form $V/C$  where $V$ is an infinite cyclic subgroup of $N_{\Mod(S_{g,b}^n)}(C)$. Again since $\Mod(S_{g,b}^n)$ has no torsion, it follows that finite subgroups of $W_{\Mod(S_{g,b}^n)}(C)$ are cyclic.  

Write, for compactness, $Q=W_{\Mod(S_{g,b}^n)}(C) $. Applying Theorem \ref{thm-rk-CM} we have 
	\begin{align}
	\gdf Q \leq \max_{F\in \fin_{Q} }\{ \vcd Q + rk\{W_Q(F)\}\} = \vcd Q+ 1.  
	\end{align}
	We will give a bound for $\vcd Q$.  Consider the short exact sequence
	\begin{align}
	\xymatrix{
		1\ar[r]& N_{\Mod(S_{g,b}^n)[m]}(C)  \ar[r]& N_{\Mod(S_{g,b}^n)}(C) \ar[r] &K 
		\ar[r]&1 
	}, 
	\end{align}
	where $K$ is a  subgroup of  the finite group $\Aut(H_1(S_{g,b}^n, \z_m))$. Passing to the quotient,  we have
	\begin{align}
	\xymatrix{
		1\ar[r]& W_{\Mod(S_{g,b}^n)[m] }(C) \ar[r]& W_{\Mod(S_{g,b}^n)}(C) \ar[r] &K' 
		\ar[r]&1 
	}, 
	\end{align}
	where $K'\simeq K$.  Since $C$ is maximal in  $\C_{\Mod(S_{g,b}^n)[m]}$, then  $W_{\Mod(S_{g,b}^n)[m]}(C) $  is torsion-free, and thus 
	\begin{align}
	\vcd W_{\Mod(S_{g,b}^n)}(C) &= \vcd W_{\Mod(S_{g,b}^n)[m]}(C)  \nonumber  \\
	&\leq  \gdf W_{\Mod(S_{g,b}^n)[m]}(C)  \nonumber  \\
	& \leq \gdf  N_{\Mod(S_{g,b}^n)[m]}(C) \label{ineq-thm5-8}\\
	&\leq\gdf  \Mod(S_{g,b}^n) \label{ineq-inclus}\\
	\end{align}
	where the equality holds since $\Mod(S_{g,b}^n)[m]$ has finite index in $\Mod(S_{g,b}^n)$,   inequality (\ref{ineq-thm5-8})  is given  in the proof of   \cite[Thm. 5.8]{luck-weiermann}, and inequality (\ref{ineq-inclus}) follows from  subgroup inclusion. Thus the result follows. 
\end{proof}

\end{lem}

Finally, we have the desired bound for surfaces with boundary:

\begin{prop}
Let $b\geq 1$. If $g=0$,  suppose $b+n\geq 4$, and if  $g\geq 1$,  suppose $2g+b+n\ge 3$.  Then $$\gdvc \Mod(S_{g,b}^n)\leq \gdf \Mod(S_{g,b}^n)+1,$$ and equality holds if  $g\in\{0,1\}$. 
\end{prop}
\begin{proof}
For the inequality $\gdvc \Mod(S_{g,b}^n)\leq \gdf \Mod(S_{g,b}^n)+1$,  the proof is again an immediate consequence of Proposition \ref{prop:main}, Remark  \ref{rem-prop-max} and  Lemma \ref{lem-gdf-w-b}, with $d= \gdf \Mod(S_{g,b}^n)$. 

By \cite[Thm. 4.1]{Harer}, $\vcd \Mod(S_{g,b}^n)$  and the maximal rank of  an abelian  subgroup  of $Mod(S_{g,b}^n)$ are equal if only if $g\in \{0,1\} $.  Let $\Lambda$ that abelian subgroup of rank $\vcd \Mod(S_{g,b}^n)$, then $\gdvc \Lambda = \vcd \Mod(S_{g,b}^n)+1$, therefore $   \gdvc \Mod(S_{g,b}^n) \geq \vcd \Mod(S_{g,b}^n)+1$. 
 \end{proof}


\begin{thebibliography}{} 

\bibitem{AMP}{J. Aramayona and C. Mart\'inez-P\'erez. The proper geometric dimension of the mapping class group. Algebraic and Geometry Topology 14 (2014)217-227}



\bibitem{birman}{J. S. Birman. Braids, links and mapping class groups. Ann. Math. Stud. 82, Princeton University Press, 1974. }

\bibitem{bonatti-paris}{C. Bonatti and L. Paris. Roots in the mapping class groups. Proc. Lond. Math. Soc. (3)  98  (2009),  no. 2, 471--503.}

\bibitem{Cameron} P. J. Cameron, R. Solomon, A. Turull, Chains of subgroups in symmetric groups.
{\em J. Algebra} 127 (1989).


\bibitem{DP}{D. Degrijse, N. Petrosyan. Bredon cohomological dimension for groups acting on CAT(0)-spaces. Group. Groups Geom. Dyn.  9  (2015),  no. 4, 1231--1265.   }

\bibitem{DP2} D. Degrijse, N. Petrosyan.Geometric dimension of groups for the family of virtually cyclic subgroups.
{\em J. Topol.} 7 (2014) .

\bibitem{margalit}{B. Farb and D. Margalit. A primer on mapping class groups. Princeton Mathematical Series, 49, Princeton Univ. Press, Princeton, NJ, 2012. }



\bibitem{FG} R. Flores and J. Gonz\'alez-Meneses. Classifying spaces for the family of virtually cyclic subgroups of braid groups. Preprint, {\em arXiv:1611.02187}.

\bibitem{MF-BN} M. G. Fluch and B. E. A. Nucinkis.  On the classifying space for the family of virtually cyclic subgroups for elementary amenable groups. 
Proc. Amer. Math. Soc.  141  (2013),  no. 11, 3755--3769.




\bibitem{Harer}{J. L. Harer. The virtual cohomological dimension of the mapping class group of an orientable surface. Inventiones mathematicae, 84, 157-176(1986) Springer-Verlag. }

\bibitem{HOP} S. Hensel, D. Osajda, P.  Przytycki, Realisation and dismantlability. {\em Geom. Topol.} 18 (2014).


\bibitem{Ivanov} N. V. Ivanov. Mapping class groups. {\em Handbook of geometric topology}, North-Holland, Amsterdam, 2002. 

\bibitem{Daniel-Leary}{D. Juan-Pineda and I. J. Leary. On classifying spaces for the family of
virtually cyclic subgroups. In Recent developments in algebraic topology,
volume 407 of Contemp. Math., pages 135-145. Amer. Math. Soc., Providence, RI, 2006. }


\bibitem{Daniel-Ale}{D. Juan-Pineda and A. Trujillo-Negrete. Dimension for classifying spaces for the family of virtually cyclic subgroups in mapping class groups}. To appear in
Pure and Applied Mathematics Quarterly.


\bibitem{KPS} A. Karrass, A. Pietrowski, D. Solitar, Finite and infinite cyclic extensions of free groups. {\em J. Austral. Math. Soc. 16 (1973).} 



\bibitem{luck-cat(0)-vc}{W. L\"uck. On the classifying space of the family of virtually cyclic subgroups for CAT$(0)$-groups. M\"unster J. of Math. 2, 201-214, 2009}

\bibitem{luck}{W. L\"uck. Survey on classifying spaces for families of subgroups. In infinite groups: geometric, combinatorial and dynamical aspects, volume 248 of Progr. Math., pages 269-322. Birkh\"auser, Basel, 2005. }




\bibitem{luck-type}{W. L\"uck. The type of the classifying space for a family of subgroups. Journal of Pure and Applied Algebra 149  177-203, 2000. }





\bibitem{luck-weiermann}{W. L\"uck., and M. Weiermann. On the classifying space of the family of virtually cyclic subgroups,
Pure and Applied Mathematics Quarterly, Vol. 8 Nr. 2 (2012). }

\bibitem{Conchita}{ C. Mart\'inez-P\'erez, A bound for the Bredon cohomological dimension
{\em J. Group Theory}, 731-747, 10  (2007).}




\bibitem{Wolpert} S. A. Wolpert. Geometry of the Weil-Petersson completion of Teichm\"uller space. IN {\em Surveys in differential geometry, Vol. VIII}. Int. Press, Somerville, MA, 2003.

 \end{thebibliography}
\end{document}